\documentclass[a4paper]{amsart} 
\usepackage{amsmath}
\usepackage{amssymb}
\usepackage{verbatim}
\usepackage{amsthm}
\usepackage{mathrsfs}
\usepackage{graphicx}
\usepackage{mathtools}
\usepackage[colorlinks,pdfdisplaydoctitle,linkcolor=blue,citecolor=blue]{hyperref}
\usepackage{caption}
\usepackage{subcaption}
\usepackage{url}
\usepackage{epstopdf}
\usepackage{pgf,tikz}
\usepackage{bbm}
\usepackage[notref,notcite,color]{showkeys}
\usepackage{extarrows}
\usetikzlibrary{arrows}
\usepackage{algorithm}
 \usepackage{algorithmicx}
\usepackage{amssymb,amsthm}    
\usepackage{graphicx}          
\usepackage{amsmath}         
\usepackage{caption}
\usepackage{subcaption}
\usepackage{listings}
\usepackage{bbm}
\usepackage{color}
\usepackage{ dsfont }
\usepackage{enumerate}
\usepackage{todonotes}

\renewcommand{\Delta}{\triangle}

\definecolor{darkblue}{rgb}{0,0,0.7}

\definecolor{darkgreen}{rgb}{0.01,0.75,0.24}
\newcommand{\bbR}{\mathbb{R}}

\def \Ee[#1]{\mathcal{E}^{\text{{#1}}}}
\def\R{\mathbf{R}}

\def\pa[#1,#2]{\frac{\partial {#1}}{\partial {#2}} }

\def\idom[#1,#2,#3]{\int_{#1}\hspace{1pt} {#2} \hspace{1pt} \text{d}{#3}}
\def\res[#1,#2]{\left.{#1}\right|_{#2}}

\def\var[#1,#2]{\langle \delta \mathcal{E}^{\text{{#1}}}({#2}),v\rangle}
\def\vars[#1,#2,#3]{\langle \delta^2\mathcal{E}^{\text{{#1}}}({#2})v,{#3}\rangle}
\def\vard[#1,#2,#3,#4]{\langle \delta\mathcal{E}^{\text{{#1}}}({#2})-\delta\mathcal{E}^{\text{{#3}}}({#4}),v\rangle}

\newcommand{\PP}{\mathbb{P}}
\newcommand{\C}{\mathcal{C}}

\newcommand{\dd}{ \,\textrm{d}}

\newcommand{\Tr}{\mathrm{Tr}}

\newcommand{\be}{\begin{equation}}
\newcommand{\en}{\end{equation}}
\newcommand{\ben}{\begin{equation*}}
\newcommand{\enn}{\end{equation*}}
\newcommand{\bea}{\begin{aligned}}
\newcommand{\ena}{\end{aligned}}

\def\ba#1\ena{\begin{align}#1\end{align}}
\def\ban#1\enan{\begin{align*}#1\end{align*}}

\theoremstyle{plain}
\newtheorem{thm}{Theorem}[section]
\newtheorem{defn}[thm]{Definition}
\newtheorem{conj}[thm]{Conjecture}
\newtheorem{lem}[thm]{Lemma}
\newtheorem{cor}[thm]{Corollary}
\newtheorem{prop}[thm]{Proposition}
\newtheorem{assumption}[thm]{Assumption}

\newtheorem{remark}[thm]{Remark}

\numberwithin{equation}{section}

\begin{document}

\title[A statistical framework and analysis for pulse compression]{A  statistical framework and analysis for perfect radar pulse compression}
\author[N. K. Chada] {Neil K. Chada}
\address{\textcolor{black}{Department of Mathematics, City University of Hong Kong, 83 Tat Chee Ave, Kowloon Tong, Hong Kong}}
\email{neilchada123@gmail.com}


\author[P. Piiroinen] {Petteri Piiroinen}
\address{Department of Mathematics, University of Helsinki, FI-00014, Helsinki, Finland}
\email{petteri.piiroinen@helsinki.fi}

\author[L. Roininen] {Lassi Roininen}
\address{School of Engineering Science, Lappeenranta-Lahti University of Technology, Yliopistonkatu 34, FI-53850 Lappeenranta, Finland}
\email{lassi.roininen@lut.fi}


\begin{abstract}
Perfect radar pulse compression coding is a potential emerging field which aims at providing rigorous analysis and fundamental limit radar experiments.
It is based on finding non-trivial pulse codes, which we can make
statistically equivalent, to the radar experiments carried out with
elementary pulses of some shape. A common engineering-based radar
experiment design, regarding pulse-compression, often omits the rigorous
theory and mathematical limitations. In this work our aim is to develop
a mathematical theory which coincides with understanding the radar
experiment in terms of the theory of comparison of statistical
experiments. We review and generalize some properties
of the It\^{o} measure. We estimate the unknown i.e. the structure
function in the context of Bayesian statistical inverse problems.
We study the posterior for generalized $d$-dimensional inverse problems,
where we consider both real-valued and complex-valued inputs
for  posteriori analysis. Finally this is then extended to
the infinite dimensional setting, where our analysis suggests the underlying posterior is non-Gaussian.
\end{abstract}

\maketitle

\section{Introduction}
\label{sec:introduction}
\bigskip
\textcolor{black}{Developing mathematical theory of comparison of statistical measurements
is crucial for understanding fundamental limits of radar experiments
\cite{LH96, LM04,P05,MIS90}. In the specific field of radar coding, one
is interested in studying modulation patterns of transmitted radar
signals. We are interested in pulse compression coding of coherent
scatter radar experiments, where coding schemes play a crucial role in
achieving a high range resolution (a radar terminology used to distinguish different signals of pulses).
Pulse compression is a popular approach aimed at increasing the range resolution, through reducing the width of various pulses
but increasing the length, or amplitude.
{Pulse codes are a common approach to
modelling the underlying target function, which can be thought of as
concentrated length pulses with constant amplitude and phase.  
The flexibility and choices of the amplitude and frequency, has motivated various choices for pulse codes. Arguably one of the
most common example are binary phase codes which omit a constant
amplitude between two phases $\phi \in \{-1,1\}$. Other examples of codes
include Barker codes \cite{RB53} and alternating codes
\cite{RB53,MSL86,MN97}. The accuracy of the estimated target
function, i.e. the scattering function as used in radar modelling,
depends hugely on the pulse
compression design.  There is a rich literature on coding techniques,
see e.g. \cite{MSL14a,MSL14b,MSL86,LH96,V12}, that discusses how to
best optimize radar experiments with various compression techniques and
assumptions.} \textcolor{blue}{The focus of this work is on perfect radar pulse compression, which is based on pulse compression using \textit{perfect codes}, which we developed by
Lehtinen et al. \cite{LDPO09} to remove high frequencies, or sidelobes of the pulse. Specifically, perfect codes are codes with a shape, referred to as a pulse, whose sequence is a single elementary pulse. By this
we mean a pulse with compact support. An example of this would be a bump or triangular function.}}
Given the complexity of these experiments it is important
to understand, through a mathematical, and statistical, framework, how
we can best formulate these experiments and gain an understanding from
them.

Given the level of uncertainty that can arise within radar coding, a
useful way to tackle these issues is through a statistical
understanding. The work of Lehtinen \cite{MSL86} first considered this
problem by modelling the scattering measurements within the signal as a
statistical inverse problem \cite{KS05, AMS10}. In other words we could
characterize our signal through noisy measurements. With this work an
important assumption was taken regarding the signal, which is that it is
normally distributed. This assumption was made both for practical
purposes but also that many signals omit a pulse form similar to a
Gaussian density or kernel. Since this initial development there has
been a number of papers looking to extend these results in a more
rigorous fashion. Much of the current literature has considered a
comparison of statistical measurements. This has lead to various pieces
of work which have adapted ideas from Le Cam Theory, notably the work by
Piiroinen et al. \cite{LDPO09,P05, RLPV14}. Other fundamental questions
that have been considered in this context is how one can optimize the
baud length of the radar. The baud length can be described as the time
step which is used to discretize the radar signal. Numerically this was
tested in the work of \cite{LD13} which looked at the simple case for
optimizing the baud length to minimise the posteriori variance. This was
shown only in the context of specific targets.

Our motivation behind this work is to bridge the gap between the various
communities in radar coding, namely by deriving a first simplified
Bayesian statistical analysis for perfect radar pulse compression. In particular
we aim to build upon the current theory and develop a better
understanding of statistical properties through characterizing a
posterior distribution of the radar signal. The underlying mathematics of the posterior
signal and its properties pose intriguing questions, such as \textcolor{black}{whether
the posterior is a Gaussian distribution}, and understanding this
for high and infinite dimensions. This question will act as the motivation
behind this work.

\subsection{Contributions}
The following bulletpoints summarize the contributions of this work.
\begin{itemize}
\item To the best of our knowledge this is the first paper \textcolor{black}{focused on} deriving a
  statistical framework, and analysis, for the theory of perfect radar pulse
    compression. Our framework will be largely based on the notion and
    generalization of It\^{o} measures to scattering functions.
\item We aim to analyze perfect radar pulse compression in a Bayesian setting. This motivates studying and understanding statistical properties of our scattering function.
\textcolor{blue}{We aim to form a posterior distribution of the variance of the scattering function}.  We first consider a $d$-dimensional case, where $d< \infty$. \textcolor{blue}{Furthermore we also provide a result related to showing whether two posterior variances coincide, of two signals, with different waveforms}.
This will be considered for both real valued and complex valued values.
To conclude our analysis we consider the $d$-dimensional
  setting, for $d=\infty$, where we show our underlying posterior is non-Gaussian which follows an inverse Wishart distribution. Here we use the notion of rapidly decreasing functions for our function spaces setting, to characterize the posterior.
\item We discuss and review a number of key open questions which are
  still very much at the core of this field. These problems are
    motivated through both a mathematical and engineering perspective.
    Much of these questions follow on from the results obtained in this
    work.
\end{itemize}

\subsection{Outline}
Our work will be split into the following sections: we begin Section
\ref{sec:radar_coding} with a review of radar signaling, and in particular
pulse compression. Section \ref{sec:posterior_analysis} will be
dedicated to understanding posterior distribution of the signal that is
defined through the previous section, which highlights our main results.
 Appendix  \ref{sec:d_analysis} and
\ref{sec:inf_analysis} will be devoted to the analysis of the
$d$-dimensional and infinite-dimensional analysis, which ultimately shows the proof of our main theorem. Finally
we review and discuss a number of questions still
to be answered while concluding our findings, in Section
\ref{sec:conclusion}.

\section{Radar coding}
\label{sec:radar_coding}

\textcolor{black}{In this section we will provide a brief background
review on radar modelling. We will introduce the concepts of an It\^{o}
measure, which is what we base our signal on}, \textcolor{blue}{however we will
postpone the mathematically rigorous definition in the appendices}. We
will provide a number of useful definitions while stating our main model
form we consider for our signal.

\textcolor{black}{
Within radar modelling, one is concerned with the sending and receiving
of a signal} which, depending on the task at hand, can take
different representations. \textcolor{black}{The form of the signal we take is based on
It\^{o} measures. This concept is explained through the following
definitions. We will give a preliminary and somewhat vague definition
first and properly define these in the appendices where we define
complex Gaussian measures. This is to improve readability.}
\begin{defn}
\textcolor{black}{
\label{def:ito1}
  An It\^{o} measure \textcolor{blue}{$\mu$} is a complex Gaussian measure on $\mathbb{R}^n$
  with a structure measure $X$ on $\mathbb{R}^n$ given by
  \[
    X(B_1 \cap B_2) = \mathbb{E}[\mu(B_1)\overline{\mu(B_2)}],
  \]
  for every Borel set $B_1$ and $B_2$. The structure measure is uniquely
  determined with a \emph{variance function}
   \[
     X(B) = \int_B {\lvert {\sigma(x)} \rvert}^2 \mathrm{d}x.
   \]}
\end{defn}

\textcolor{blue}{Later we will consider a special case of constant variance.
For example, a complex white noise process has a constant variance.}
\textcolor{blue}{
\begin{defn}[It\^{o} measure with constant variance]
\label{def:const}
  We say that covariance structure $X(\sigma)$ has a constant variance
  $ |\,\sigma\,|^2$ if the variance function $|\,\sigma(x)\,|^2 =
  \sigma_0^2 > 0$ for every $x$.
\end{defn}}


  \textcolor{black}{The property we used in the above definition of an It\^{o}
  measure is known as incoherence.}
\textcolor{blue}{An It\^{o} measure model is an incoherent scatter radar signal
(time-coherent, and spatial-incoherent signals). The concept of
coherence comes from physics, which implies that two waves, or signals
can interfere with each other. As we are modelling spatial-incoherent
signals, this implies in the spatial dimension, the signals do not
interfere and are independent.}

In radar modelling the scattered signal $z$ from an \textcolor{blue}{It\^{o} measure $\mu$}
can be express by an convolution of
  some transmission envelope $\epsilon^q(t)$
\textcolor{black}{(described as the shape or amplitude)},
known as
an  It\^{o} integral scattering relation, which is given as
\begin{equation}
\label{eq:og}
  z^q(t) = \int_{\mathbb{R}^3} \epsilon^q(t-S(r))\mu^q(\mathrm{d}r) + \sqrt{T}\xi^{q}(t),
\end{equation}
\textcolor{blue}{where $q$ is a repetition index of the experiment to
  facilitate possibly different modulations in different repetitions.
The notation $S(r)$ denotes the total travel time of the signal from the
transmission, through to the scattering point $r$ to the receiver. This
implies that \eqref{eq:og} sums up all elementary scatterings which
 takes into account the phase of the signal.
The final term is related to thermal noise, where $T$ denotes the
temperature and $\xi^q$ is assumed to be complex Gaussian white noise. }
\textcolor{blue}{\begin{remark}
Throughout the paper we will use different terms to refer to $\epsilon^q(t)$, such as the code, or potentially the pulse of the code, which is related to the shape of the code. We note that these exact definitions are not required, and thus we omit them. However we refer the reader to \cite{RLPV14}.
\end{remark}}
\textcolor{black}{In the
  radar coding community the $\mathrm{d}r$ is usually written as
  $\mathrm{d}^3r$ to signify the fact that the integration is over
  three-dimensional space and the integral is written three times.}

Using a more mathematical way of expression this is that for every
  elementary event $\omega$ from the underlying probability
  space, the $\mu^q(\cdot, \omega)$ is a time-stationary realization of
  the random measure and the single realization of the scattered signal
  is
\begin{equation*}
  z^q(t, \omega) = \int_{\mathbb{R}^3} \epsilon^q(t - S(r))\mu^q(\mathrm{d}r, \omega) + \sqrt{T}\xi^{q}(t, \omega),
\end{equation*}
which must be understood in a generalized sense, since the realizations of
the noise $\xi^{q}(\cdot, \omega)$ and $\mu^q(\cdot, \omega)$ are both
proper measures that do not have point values.
For simplicity we can assume $q=1$ for this work related to our theory.
We keep to this unconventional notation, as it is consistent with the
field of statistical pulse compression \cite{MSL14a,MSL86}. While \eqref{eq:og} holds for a wide class of transmissive and receptive
antennas, in this work we consider a slightly different model. For
simplicity we will assume that we have a mono-static single beam radar.
To be more precise, if the back and forth signal time along the beam is
denoted by $r$, then $S(r)=r$ and we describe the signal model as a
one-dimensionl convolution integral equation along to beam
\begin{equation}
\label{eq:model}
  z^q(t) = \int_{\mathbb{R}} \epsilon^q(t-r)\mu^q(dr) + \sqrt{T}\xi^q(t)
       = \epsilon * \mu^q(t) + \sqrt{T}\xi^q(t).
\end{equation}
\textcolor{black}{%
As previously stated, this could be written more rigorously and must be
understood in a generalized sense, for instance via temperate
distribution valued random objects. The structure function describing
the spatial correlations of the target It\^o measure is $X =
X(\sigma)$. Explicitly, this can be given directly describing the action
of the measure as
\begin{equation}
\label{eq:sp}
  \int_{\mathbb{R}^2} \phi(r, r') \langle \mu(r),\overline{\mu(\mathrm{d}r')} \rangle
  = \int_{\mathbb{R}} \phi(r, r)  X(\mathrm{d}r)
  = \int_{\mathbb{R}} \phi(r, r)  {\lvert {\sigma(x)} \rvert}^2 \mathrm{d}r,
\end{equation}
where $\phi$ is any smooth enough test function.
The incoherence assumption corresponds to the model where
the scatterings from disjoint volumes are mutually statistically
independent.}
Similarly, the temporal correlation of
the noise can be given as
\begin{equation}
\label{eq:sp_n}
  \int_{\mathbb{R}^2} \phi(t, t') \langle \xi^q(t),\overline{\xi^q(\mathrm{d}t')} \rangle
  = \int_{\mathbb{R}} \phi(t, t)  \mathrm{d}t.
\end{equation}
so the correlation structure has a \emph{constant} variance function
${\lvert {\sigma(x)} \rvert}^2 = 1$ and $\phi$ is a smooth enough test
function.
Using \eqref{eq:sp} and \eqref{eq:sp_n} \textcolor{black}{we can compute the lag estimate, or covariance, of the measurements as
\begin{align*}
  \int_{\mathbb{R}^2} & \phi(t, t') \langle
  z^q(\mathrm{d}t),\overline{z^q(\mathrm{d}t')} \rangle \\
    &= \int_{\mathbb{R}^4}
    \phi(t, t') \epsilon^q(t - r) \overline{\epsilon^q(t' - r')}
    \langle \mu(\mathrm{d}r),\overline{\mu(\mathrm{d}r')} \rangle
    \mathrm{d}t \mathrm{d}t'
    + T\int_{\mathbb{R}} \phi(t, t) \mathrm{d}t \\
    &= \int_{\mathbb{R}^3} \epsilon^q(t-r) \overline{\epsilon^q(t'-r)}
    X(\mathrm{d}r)
    \mathrm{d}t \mathrm{d}t'
    +
    T\int_{\mathbb{R}} \phi(t, t) \mathrm{d}t \\
    &= \int_{\mathbb{R}^2}
    \mathrm{d}t \mathrm{d}t' \phi(t, t')
    \int_{\mathbb{R}} A_{tt'}(r)|\sigma(r)|^2 \mathrm{d}r
    + T \int_{\mathbb{R}} \phi(t, t) \mathrm{d}t
\end{align*}}
where $A_{tt'}(r) =  \epsilon^q(t-r) \overline{\epsilon^q(t'-r)}$,
assuming that we can quite freely change the orders of integration and
that the noise is indendent from the signal. Usually this is written
distributional sense as
\begin{align*}
  \langle z^q(t),\overline{z^q(t')} \rangle
    &= \int_{\mathbb{R}} A_{tt'}(r)|\sigma(r)|^2 \mathrm{d}r +
    T\delta_0(t - t')
\end{align*}
where $\delta_0$ stands for the Dirac point mass at origin.
\textcolor{black}{This latter formalism was introduced by Van Trees's book on
‘\textit{Detection, Estimation and Modulation theory}’ \cite{HVT71} ,
but it has unfortunately not been really exploited in radar literature.
It is not complicated, and derivations can be made rigorous and simple. We
refer the reader here for further details on these derivations.}


\textcolor{black}{Both \eqref{eq:og} and \eqref{eq:model} assume that we have a time-independent signal model, whereas in the case if the signal was time dependent our signal would be modified to
\begin{equation}
\label{eq:time}
z^q(t) = \int^{\infty}_{0} \epsilon^q(t-r)\mu^q(dr;t) + \sqrt{T}\xi^q(t),
\end{equation}
so that now $t$ can be treated as either the scattering time or the
reception time. Our analysis can be generalized to the time-dependent case, but for simplicity we focus on models of the form in Eqn. \eqref{eq:og} and \eqref{eq:model}.} Our quantity of interest in this
model is the signal denoted by $\mu(\cdot)$. In radar signaling this
unknown we are aiming to estimate is known as an incoherent scattering
target of a time-coherent signal. A fundamental question that arises is how to best estimate or
model the underlying signal? We will make the following assumption, but
we will refer it explictly when it is actually used.

\begin{assumption}
\label{def:meas}
Assume we have two measurements defined as
\begin{align}
\label{eqn:m1}
  z_1  &= \epsilon_1 * \mu(\sigma)+ \xi_1, \\
\label{eqn:m2}
  z_2  &= \epsilon_2 * \mu(\sigma) + \xi_2,
\end{align}
where $\xi_1 \sim \xi_2$ are of a complex Gaussian form, $\epsilon$ is a
  transmitted waveform and $\mu(\sigma)$ is the It\^{o} measure scatterer
  such that its structure measure $X$ depends on the given variance
  function $\sigma$.
\end{assumption}


\section{Bayesian Posterior Analysis}
\label{sec:posterior_analysis}
\textcolor{black}{In this section we provide a statistical analysis on signals arising
from perfect radar pulse compression. In particular the focus will be on
understanding the posterior distribution of $\sigma$. 
The derived analysis will form a basis for the higher and infinite
dimensional setting, in succeeding sections.}
\textcolor{black}{In all the definitions, what is meant}, by densities
and conditioning of the generalized random variables are reviewed in the
Appendix. By the posterior distribution we mean the regular conditional
distribution of the generalized random variable given the data random
variable. 
Specifically the characteristic functions are defined in Appendix \ref{sec:char},
and the densities are defined in Appendix \ref{sec:den}.

\textcolor{black}{In order to study the posterior distribution of \emph{formal standard
deviation} function $\sigma$, instead of the actual variance function
$|\sigma|^2$, we have to express the fully hierarchical Bayesian model
that corresponds to the problem at hand. Before we discuss our Bayesian hierarchical model, 
we note that when we write
\[
  z = \epsilon * \mu(\sigma)+ \xi,
\]
and assume that $\mu(\sigma)$ is an It\^{o} measure scatterer such that its
structure measure $X$ depends on the given the formal standard deviation
function $\sigma$. One may think that we are given the conditional
distribution of the signal $z$ given the doubly stochastic $\mu(\sigma)$, i.e.,
\[
  z \,|\,  \mu, \sigma \; \dot= \; \epsilon * \mu(\sigma)+ \xi.
\]
However, we cannot directly observe the formal standard deviation
function, i.e. there is a hierarchical Bayesian connection
\[
  z \,|\,  \mu, \sigma \; \dot= \; z \,|\,   \mu.
\]
This is equivalent with the fact that $\sigma$ and $z$ are conditionally
independent given $\mu$. In order to specify that $\mu = \mu(\sigma)$ is an It\^{o} measure 
with structure function given $\sigma$, we mean that we are given the conditional
distribution of $\mu$ given $\sigma$:
\[
  \mu | \sigma \text{ is an It\^{o} measure with variance function $|\sigma|^2$},
\]
and finally we give a prior distribution for the formal standard
deviation function $\sigma$, which is denoted as $\pi$.
The scatterer $\mu$ can thus be seen as a nuisance parameter in this posterior analysis.}
\textcolor{black}{
In order to arrive to the main theorem of the paper, let us first
consider discrete versions of this. Suppose that the space is
discretized into a finite set of points. Under this assumption, the
hierarchical model becomes
\[
  \begin{cases}
    \underline{z} \,|\, \underline{\mu} & \sim N_d(A \underline{\mu},\; T
    \mathrm{I}_d), \\
    \underline{\mu} \,|\, \underline{\sigma} & \sim N_d(0,\;
    \mathrm{diag}(|\,\underline{\sigma}^2\,|).
  \end{cases}
\]
The discretization would turn the It\^{o} measures into finite dimensional
random vectors \( \underline{z} = (z_1, \dots, z_d) \), 
\( \underline{\mu} = (\mu_1, \dots, \mu_d) \) and also turn the variance
function into a finite dimensional random vector
\( |\,\underline{\sigma}^2\,|) = ( |\, \sigma_1^2 \,|, \dots, |\,
\sigma_d^2 \,|) \). The convolution corresponds to a matrix $A$. 
If we assume that the variance function is constant, that can now be
understood as $\sigma_i = \sigma_0$ for every $i = 1, \dots, d$ and the
model becomes fully pooled model. In general, this means that the
structure measure is randomized with just a single random number
(a single complex valued random variable).
The fully pooled discrete Bayesian model is therefore
\[
  \begin{cases}
    \underline{z} \,|\, \underline{\mu} & \sim N_d(A \underline{\mu},\; T
    \mathrm{I}_d), \\
    \underline{\mu} \,|\sigma_0 & \sim N_d(0,\;
    |\sigma_0^2\,| \mathrm{I}_d), \\
    \sigma_0 &\sim \pi,
  \end{cases}
\]
where $\pi$ is the prior distribution we choose for $\sigma_0$.
Since the model has the implicit conditional independence assumption, i.e.,
\[
  \underline{z} \,|\, \underline{\mu}, \sigma_0 \; \dot= \; 
  \underline{z} \,|\, \underline{\mu},
\]
we can first consider $\sigma_0$ given and fixed, and we arrive to
a well-known simple Bayesian model
\[
  \begin{cases}
    \underline{z} \,|\, \underline{\mu} & \sim N_d(A \underline{\mu},\; T
    \mathrm{I}_d), \\
    \underline{\mu} & \sim N_d(0,\; |\sigma_0^2\,| \mathrm{I}_d).
  \end{cases}
\]
The marginal distribution of the discrete signal $\underline{z}$
satisfies
\[
  \mathbb{E}(e^{i\underline t' \underline z}) 
  = \mathbb{E} \big(\mathbb{E}(e^{i\underline t' \underline z} \,|\,
  \underline{\mu}) \big)
  = \dots = \exp(-\frac12 \underline t'(T \mathrm{I}_d + |\sigma_0|^2
  AA') \underline t),
\]
where $\underline t \in \R^d$ and $A'$ stands for the Hermitean adjoint
of the matrix $A$. Therefore, we see that unconditionally
\[
  \underline{z} \sim N_d(0, \; \Sigma),
\]
such that $\Sigma = T \mathrm{I}_d + |\sigma_0|^2 AA'$ given we know the
value of $\sigma_0$. Repeating the previous we observe that this leads
to a Bayesian model
\[
  \begin{cases}
  \underline{z} \,|\, |\sigma_0|^2 & \sim N_d(0, \; \Sigma), \\
    |\sigma_0|^2 & \sim \pi.
  \end{cases}
\]
Provided that $AA' > 0$ is positive definite, then $\Sigma$ and
$|\sigma_0|^2$ are bijective affine transforms of each other and we can
therefore give the prior to $\Sigma$ instead. It is well-known that the
conjugate prior for the covariance matrix of centered multivariate
normal distribution is the inverse Wishart  distribution. A definition of
such a distribution is provided below.}
\textcolor{black}{ 
\begin{defn}[Inverse Wishart distribution]
\label{def:IWD}
A $p \times p$-dimensional random matrix $X \sim
\mathcal{W}^{-1}(\Psi, \nu)$ has the inverse Wishart distribution with
$p \times p$ positive definite scale matrix $\Psi$ and $\nu > p - 1$
degrees of freedom if its density function is
\[
\pi(\Sigma) 
  = \frac{|\Psi|^{\nu/2}|\Sigma|^{-(\nu + p + 1) / 2}}
         {2^{\nu p/2} \Gamma_p(\frac{\nu}{2})}
    \exp\Big(-\frac{\Tr(\Psi\Sigma^{-1})}{2}\Big),
\]
where $\Gamma_p$ is the $p$-variate Gamma function, and $\Tr(\cdot)$ denotes
the trace of the matrix. The $p$-variate Gamma function is defined as a
generalization of Gamma function where the positive number $s > 0$ is
replaced with a positive definite $p \times p$ matrix and that is
numerically equivalent with
\[
  \Gamma_p(s) = \pi^{p(p - 1) / 4}\prod_{j = 1}^p \Gamma(s - (j - 1)/2).
\]
for $s > (p - 1) / 2$.
\end{defn}}
\textcolor{black}{To help visualize this difference with a Gaussian
distribution we plot three different density functions of the inverse
Wishart distribution. This is presented in Figure \ref{fig:IW}.}
\begin{figure}[h!]
\centering
\includegraphics[scale=0.5]{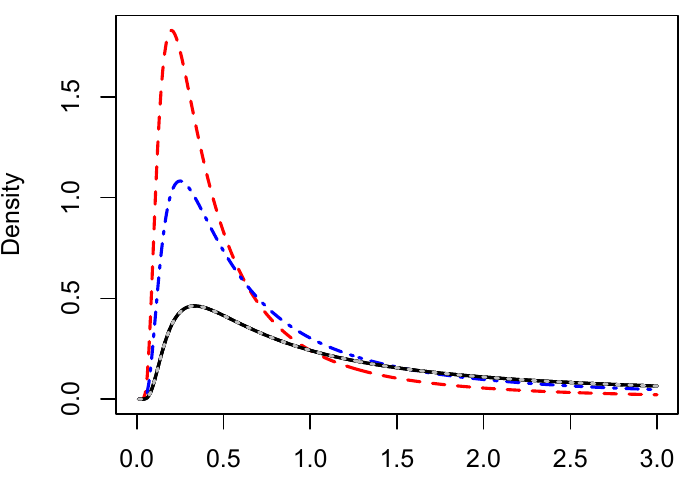}
\caption{\textcolor{black}{Various density plots of the inverse Wishart distribution with varying degrees of freedom. Red plot is for 3 degrees of freedom. Blue plot is for 2 degrees of freedom. Grey plot is for 1 degree of freedom}.}
\label{fig:IW}
\end{figure}
\\
Namely, if we assume that
\[
  \begin{cases}
  \underline{z} \,|\, \Sigma & \sim N_d(0, \; \Sigma), \\
    \Sigma & \sim \mathcal W^{-1} (\Psi, \nu),
  \end{cases}
\]
where $\Psi > 0$ is positive definite $d \times d$ matrix and $\nu > 0$,
then the posterior distribution of $\Sigma$ is
\[
  \Sigma \,|\, \underline{z} \sim \mathcal W^{-1} (\underline z \underline z' + \Psi, \nu + 1),
\]
where $\underline z \underline z'$ is the $d \times d$ rank one matrix
obtained as the outer product of the vector $\underline z$. Moreover, it
is well known that chi squared distribution is not a conjugate distribution
for this likelihood, i.e. if we would assume that $\sigma_0$ has a
normal distribution and therefore an affine transform of $|\sigma_0|^2$
would have a chi squared distribution, the posterior would not be an
affine transform of a chi squared, let alone normal.

More generally, if there are $k$ different random numbers in the
discrete formal standard deviation, i.e. then we consider the hierarchical model
\[
  \begin{cases}
    \underline{z} \,|\, \underline{\mu} & \sim N_d(A \underline{\mu},\; T
    \mathrm{I}_d), \\
    \underline{\mu} \,|\sigma_0 & \sim N_d(0,\;
    \mathrm{diag}(|\sigma^2\,|)), \\
    \Sigma &\sim \mathcal W^{-1}(\Psi, \nu),
  \end{cases}
\]
where $\Sigma = T \mathrm{I}_d + A \mathrm{diag}(|\sigma^2\,|) A'$, then
the posterior distribution of $\Sigma$ given the measured discretized
signal is
\[
  \Sigma \,|\, \underline{z} \sim \mathcal W^{-1} (\underline z \underline z' + \Psi, \nu + 1).
\]

We now present our main theorem of the paper, which is the
characterization of the posterior variance, related to the scattering
function. This is given through the
\textcolor{black}{following result}.

\begin{thm}
\label{thm:main}
Assume the priori distribution of $|\,\sigma\,|^2$ is interpretable as
an affine transform of inverse Wishart distribution, then the
posteriori distribution of $|\, \sigma \,|^2$ can be interpreted as a
generalized limit of affine transforms of inverse Wishart
distributions of the similar type, \textcolor{black}{given in
Definition \ref{def:IWD}}.
\end{thm}
\begin{proof}
\textcolor{black}{The proof is in Appendix
  \ref{sec:inf_analysis}, subsection \ref{ssec:proof}.}
\end{proof}

\textcolor{black}{The above result in Theorem \ref{thm:main} highlights that the
underlying posterior is not Gaussian, even under suitable interpretation
of normality assumption for the prior. The proof follows the results
obtained from Appendix \ref{sec:inf_analysis}, which is extended from
the analysis conducted in Appendix \ref{sec:d_analysis}.}

\textcolor{blue}{Let us study the case where the convolution is with respect to a Dirac
mass, i.e. the matrix $A$ above is $cI_d$. Let us assume that
$(|\sigma_i\,|^2 = |\sigma_0^2|$ for all $i$ i.e. the variance is
constant. Then the previous gives that
\[
  z_1, \dots, z_d \,|\, |\sigma_0^2\,| \sim N(0, (T + |\sigma_0|^2),
\]
and observations are independent. Thus, if we assume $T + |\sigma_0|^2
\sim \mathcal W^{-1}(\sigma_1^2, \nu)$ for scalar $\sigma_1^2 > 0$
(corresponding to $1 \times 1$ matrix), and $\nu > 0$, then 
$T + |\sigma_0|^2 \,|\, \underline{z} \sim 
\mathcal W^{-1}(\sigma_1^2 + |\underline z|^2, \nu + d)$.}
\textcolor{blue}{
A complex Gaussian distribution prior for formal standard deviation
$\sigma$ would translate as a chi squared or gamma distribution prior
for $|\sigma_0|^2$. Allowing an affine transform for gamma prior, the
posterior (for simplest case) would be of form
\[
  p(|\sigma_0|^2 \,|\, \underline z) 
  \propto (|\sigma_0|^2 + T)^{-1/2} 
  \exp(-\frac{c}2 |\sigma_0|^2 - \frac12 (T + |\sigma_0|^2)^{-1}),
\]
which is a mix of shifted gamma and inverse gamma distributions, showing
that even in the simple case the normal distribution is not a prior for
\emph{finite} observations.}
\textcolor{blue}{
For the infinite dimensional setting the underlying spaces are taken to
be the rapidly decreasing functions, or Schwartz functions. These are
defined by $\mathscr S(\C^n)$ (or the compactly supported test functions
$\mathscr D(\Omega)$) and their dual spaces $\mathscr S'(\C^n)$ of
tempered distributions (or the distributions $\mathscr D'(\Omega)$).}
\textcolor{blue}{
If we continue with the example of constant variance and Dirac mass 
transmission envelope, and we would allow collecting unboundedly many
observations (i.e. letting $d \to \infty$). Using the simple inverse
Wishart prior $T + |\sigma_0|^2 \sim \mathcal W^{-1}(\sigma_1^2, \nu)$
with scalar $\sigma_1^2 > 0$ and degrees of freedom $\nu > 0$, the
posterior for the constant variance $T + |\sigma_0|^2 \,|\, \underline{z}_d
\sim \mathcal W^{-1}(\underline{z}_d^2 + \sigma_1^2, \nu + d)$ showing
that the degrees of freedom go to infinity. However, the embedding the
observation model discretization lattice show that $\underline{z}_d^2 =
c_1 d + \mathcal{O}(d)$ and we should scale the posterior with $c_2 d^{-1/2} +
\mathcal{O}(d^{-1/2})$ in order to obtain a It\^{o} measure with constant
$|\sigma_0|^2$ as its variance function. If we also scale the
regularization $T = c_3 d + \mathcal{O}(d)$, we can see that normalized version
\[
  Z_d = \frac{T + |\sigma_0|^2  - (c_1 +
  c_3)\sqrt d}{\sqrt d} \,|\, \underline{z}_d,
\]
has has asymptotically zero mean and constant $2c_1^2 / c_2$ variance.
Therefore, under the framework we will study in more detail in the
Appendix~\ref{sec:d_analysis} the central limit theorem gives a way to
interpret the infinite observations as certain type of rescaled Gaussian
distributions.
This is, however, an asymptotic result and the rescaling requires that
the point values are replaced with distributional averages.}

\bigskip

\textcolor{blue}{Let us continue with the case where the posterior variance is constant
(which as we noticed translates to the corresponding It\^{o} measure be randomly scaled
white noise). Our next main result, is related to characterizing a
relationship
between two signals in relation to their posterior variance
of the scattering function. Since while we cannot really have infinite
observations, these do give asymptotic estimates for densely measured
observations. Moreover, while above we used the independence coming from
an unrealistic transmission envelope, we can at least get estimates for
the second moments. We also remark that this is formulated for the
original continuous model and not for the simplified finite dimensional
discrete approximation so the techniques and definitions are made
explicit in the Appendix. This is provided through the following
theorem.}
\begin{thm}
  \label{thm_yada41}
  \textcolor{blue}{
Assume Assumption~\ref{def:meas} 
and further suppose the prior covariance structure $X(\sigma)$ with
  constant $|\,\sigma(x)\,|^2 = \sigma_0^2 > 0$ for all $x$ (see
  Definition \ref{def:const}).
  If the moduli of the Fourier
  transforms of the transmitted waveforms coincide, i.e. if
  \[
    |\, \widehat \epsilon_1 \,| =
    |\, \widehat \epsilon_2 \,|,
  \]
  as Schwartz distributions, then the posterior variances
  $\mathrm{var}(|\,\sigma_0\,|^2 | \, z_1)$ and \\
  $\mathrm{var}(|\,\sigma_0\,|^2 | \, z_2)$ of the
  $\sigma$ given $z_1$ and $z_2$ are equal, i.e.
  \[
    \mathrm{var}(|\,\sigma_0\,|^2 \, | \, z_1) = \mathrm{var}(|\,\sigma_0\,|^2 \, | \, z_2).
  \]}
\end{thm}

\begin{proof}
  We show in Appendices~\ref{sec:d_analysis} to~\ref{ssec:proof} that if
  the covariance structure of the continuous measurement model
  corresponds to a constant multiplier, then
  \[
    \mathrm{Cov}(z_j \, | \, |\, \sigma \,|^2 )
	= \phi \mapsto |\,\phi\,|^2 \big(|\sigma|^2 |\,A_{j}\,|^2 + T\big),
  \]
  where $|\, A_{j}\, |^2 = |\, \widehat \epsilon_{j} \,|^2$.  This is
  the formula~\eqref{eq:quad}. Using the
  assumption of equal moduli of Fourier transforms of the transmitted
  waveforms, we see that
  \[
    \mathrm{Cov}(z_1 \, | \, |\, \sigma \,|^2 )  =
    \mathrm{Cov}(z_2 \, | \, |\, \sigma \,|^2 ),
  \]
  and therefore by the results of Appendix~\ref{sec:d_analysis} and the
  Proposition~\ref{appendix:propB6}, we obtain that
  this implies that the conditional characteristic functions
  \[
    J_{z_1 \,| \, |\, \sigma \,|^2 } = J_{z_2 \,| \, |\, \sigma \,|^2 },
  \]
  as generalized functions. Since the prior was constant, we obtain that
  the conditional densities of $|\, \sigma \,|^2$ given $z_1$ and $z_2$
  are both following the same inverse Wishart distribution under the same
  spatial discretization. Therefore, they have the same renormalized
  discretization limits and hence also their posterior variances
  coincide.
\end{proof}

\begin{remark}
\textcolor{blue}{
To summarize our main results, Theorem \ref{thm:main} is our main theorem, which we state first related to the posterior variance. Theorem
  \ref{thm_yada41} is concerned with the posterior variances in the
  special case of having constant variance function with a constant
  prior (i.e. a limiting prior of the inverse Wishart family, the
  structure function is Gaussian given the variance function).}
\end{remark}







\section{Conclusion \& Discussion}
\label{sec:conclusion}

Pulse compression has been a cornerstone of modern applied mathematics incorporating tools from information theory, Fourier analysis and harmonic analysis. Recently statistical methodologies have gained interest most notably for enabling some form of uncertainty quantification. \textcolor{black}{The focus of this work follows in a similar fashion. Specifically, the aim is to provide a statistical understanding for perfect pulse compression}. What we showed was that, through the introduction of It\^{o} measures where we assume our signal is distributed according to a Gaussian, we were able to characterize a posterior distribution of the covariance of the signal $\sigma$. As our results suggest, the resulting posterior is indeed non-Gaussian specifically an inverse Wishart distirbution. This was achieved through analysis in both a finite-dimensional setting
 and infinite-dimensions, where we introduced Gaussian measures and the concept of Schwartz functions for our function-space setting.

 As this is the first instance in understanding perfect radar pulse
 compression in a theoretical Bayesian manner, there are numerous directions to
 take for future work. One direction to consider is to understand the
 relationship between different pulses. To do so one can consider using
 various probabilistic metrics for Gaussian measures. A natural one to
 consider is the Kullback-Liebler divergence which has been analyzed in
 infinite dimensions \cite{PSSW14,PSSW15,VDV02}. However given how this
 is not an actual metric \textcolor{black}{per se} one could consider extensions to the
 Wasserstein distance and also the Le Cam distance \cite{LC86}, which
 has been used for statistical experiments.

 Another more applied direction is to consider a better way to model the pulses as usually they take the form of box-car functions or piecewise constant functions,
where imposing Gaussian \cite{NVR18} modeling assumptions can hinder performance. Recent work has shown that $\alpha$-stable processes \cite{CLR19} can be used in place which can be used for edge-preserving inversion. This would imply the prior random field has the particular form
\begin{equation*}
U(x)= \int_{[0,1]^d} f(x,x')M(dx'), \; x\in [0,1]^d,
\end{equation*}
where
\begin{equation*}
f(x,x')=\begin{cases}
1 \text{ when }  x_i'\leq x_i  \text{ for all }  i=1,\dots, d\\
0 \text{ otherwise},
\end{cases}
\end{equation*}
and $M$ is  symmetric $\alpha$-stable random measure. An example of a non-Gaussian $\alpha$-stable process are Cauchy processes \cite{CIRS18,MRHL19,RGLM16,SCR22,STC22} which have already been tested within inverse problems.
\textcolor{black}{This could be a natural direction for using more advanced non-Gaussian priors. Note this is different to the work of this paper which was focused on the covariance. Here we are stating that one could simply modify the pulse itself such that it takes a non-Gaussian form.}

More specific to the pulse compression an important question to
quantify, is the relationship of the pulses and the temperature $T$.
Specifically what occurs in the limit $T \rightarrow 0$. For the case of
$T=0$ let us assume the code is modeled as a boxcar of width $a >0$ and
unit $L^2$ norm
$$
\epsilon(t) = \epsilon_a(t) = a^{-1/2}\chi_{[0,a)}(t).
$$
Then choosing $a=1/2N$ results in the following expression for the signal
\begin{align}
z^q(n/N+t) &= \int_0^1 \epsilon_{1/2N}(n/N+t-r(\textrm{mod} \ 1)) \,\mu^q(dr) + \sqrt{T}\xi^q \nonumber\\
&\mathrm{with}~~ 0\leq t < 1/2N ~~\mathrm{and}~~ n=0\ldots N-1, \nonumber
\end{align}
are all mutually independent and equally informative measurements of $\sigma$, each separately adding the same amount of information to $\sigma$, independent of $N$. It follows that the posteriori variance of $\sigma$ approaches 0, when $N\longrightarrow\infty$. \textcolor{black}{This differs to the consensus within the radar community, which is} that increasing radar power ( equivalent to decreasing additional noise ) will give no extra benefit after some level is reached. One will naturally benefit by choosing increasingly narrow pulses as extra power becomes available.

However for the case of $T >0$, where $T$ is close to 0, what is explained above it seems plausible that the optimal radar code might be a narrow pulse. If true then the width would approach 0 as $T\longrightarrow 0$.
\begin{conj}
For each $T$ it is possible to find an optimal code  $\epsilon_T(t)$ so that
$$\lim_{T\longrightarrow 0} \sqrt{T}\epsilon_T(t/T),$$
defines a well-defined limiting shape: a fundamental typical shape of optimal radar baud.
\end{conj}
Related to this a final direction to consider is to quantify whether the optimal code, discussed in the above conjecture is unique or not. This of course could be related to how one defined the prior form, or the scattering function.These and other directions will be considered for future work.

\section*{Acknowledgements}
The authors thank Dr. Markku S. Lehtinen, for helpful discussions and
directions for the paper. NKC is supported by an EPSRC-UKRI AI for Net
Zero Grant: “Enabling CO2 Capture And Storage Projects Using AI”, (Grant
EP/Y006143/1). NKC is also supported by a City University of Hong Kong
Start-up Grant, project number 7200809. This work of PP was also
supported by the Finnish Ministry of Education and Culture’s Pilot for
Doctoral Programmes (Pilot project Mathematics of Sensing, Imaging and
Modelling).

\appendix
\section{finite dimensional analysis}
\label{sec:d_analysis}

{In this appendix we consider a generalized setting, which is the $d$-dimensional case.
For our analysis we will consider four separate cases namely; (i) real valued Gaussian random vector,
(ii) complex valued Gaussian random vector, (iii) real valued white noise and (iv) complex valued white noise.
In order to do so we recall a number of key definitions which we will use for our analysis.
Our analysis will be based on the notion of computing means and covariances through moment generating functions.
\begin{defn} (Gaussian random vector) Assume $X := (X_1,\ldots,X_n)$ is a
  real finite dimensional random vector. We say $X$ is a
  Gaussian random vector if it can be expressed in the form
\begin{equation*}
X = \mu + A Y,
\end{equation*}
where $\mu \in \mathbb{R}^n$, $A \in \mathbb{R}^{n \times k}$ and $Y =
  (Y_1,\ldots,Y_k)$ is a vector of independent standard Gaussian random
  variables. Such a vector $Y$ is called standard multinormal random
  vector or discrete real white noise vector.
\end{defn}
\begin{defn} (Complex Gaussian random vector)  Assume
  $X=(X_1,\ldots,X_n)$ is a complex finite dimensional random vector. We
  say $X$ is a complex Gaussian random vector if it can be expressed in
  the form
\begin{equation*}
X = \mu + A Y, 
\end{equation*}
where $\mu \in \mathbb{R}^n$, $A \in \mathbb{C}^{n \times k}$ and $Y =
  (Y_1, \ldots, Y_k)$ is a discrete complex white noise vector. We say a
  complex random vector $Y = Y_R + i I_I\in \mathbb{C}^k$ is a discrete complex
  white noise, if $(Y_R, Y_I) / \sqrt{2}$ is a discrete
  $2k$-dimensional real white noise.
\end{defn}}
Let us list some properties that hold for complex and real Gaussian
vector.
\begin{prop}
  Assume that $X \in \mathbb{K}^k$ is a $k$-dimensional complex
  ($\mathbb{K} = \mathbb{C}$) or real ($\mathbb{K} = \mathbb{R}$) Gaussian random
  vector. Suppose $A \in \mathbb{K}^{n \times k}$ and $\mu \in \mathbb{K}^n$. Then
  $Z = \mu + AX$ is a $\mathbb{K}$-Gaussian random vector with
  $\mathbb{K}$-expectation
  \[
    \mathbb{E}(Z) = \mu + \mathbb{E}(X),
  \]
  and its $\mathbb{K}$-covariance matrix is
  \[
    \mathrm{Cov}(Z) = A \mathrm{Cov}(X) A',
  \]
  where $A' = A^\top$, when $\mathbb{K} = \mathbb{R}$ and $A' =
  \overline{A}^\top$, when
  $\mathbb{K} = \mathbb{C}$
\end{prop}
\begin{proof}
  The expectation of $Z$ is defined as a mapping $\phi \mapsto
  \mathbb{E} (Z' \phi)$. Since $X = \lambda + BY$ for some $\mathbb{K}$-Gaussian
  random vector, we have
  \[
    Z'\phi = \mu'\phi + (A \lambda)' \phi + Y' B'A' \phi.
  \]
  Since the expectation of $Y$ is a zero mapping, we see that
  \[
    \mathbb{E}(Z) = \mu + (A \lambda),
  \]
  where $\mu$ is identified with the mapping $\phi \mapsto \mu' \phi$.
  When $\mu = 0$ and $A$ is identity, this gives also that
  \[
    \mathbb{E}(X) = \lambda,
  \]
  so the first claim follows.

  The $\mathbb K$-covariance of $Z$ is defined as a the covariance of $W
  = Z - \mathbb{E}(Z) = AB Y$ which is in turn the mapping
  \[
    \phi \mapsto \mathbb{E}(W'\phi)'(W'\phi)
    = \mathbb{E} (\phi' WW'\phi).
  \]
  Since
  \[
    W'\phi = Y' (AB)' \phi,
  \]
  we have
  \[
    (W'\phi)'(W'\phi) = \phi' AB YY' B'A' \phi,
  \]
  This implies since the covariance of $Y$ is an $\mathbb{K}$-identity
  operator, that
  \[
    \mathrm{Cov}(Z) = AB \mathrm{Cov}(Y) B'A' = ABB'A'.
  \]
  Again, when $A$ is an identity, this gives that the covariance of $X$
  is $BB'$ so the latter claim follows.
\end{proof}
\begin{remark}
Note that this proof generalizes immediately to infinite dimensional
setting, as we will see in the succeeding section. The reason that the covariance of
$Y$ is an identity in both real and complex case is the following.

When $\mathbb{K} = \mathbb{R}$ this is well-known, however for $\mathbb{K} =
\mathbb{C}$ we can argue as follows. For any complex vector $z$ then
$z'z$ is real-valued and its real part is $z_R^\top z_R + z_I^\top z_I$, where $R$
denotes real and $I$ denotes imaginary.
Now let $X$ and $Z$ be the real and imaginary part of $Y'\phi$.
  \begin{align*}
  X = (Y'\phi)_R &= ((Y_R + i Y_I)'(\phi_R + i\phi_I))_R = Y_R^\top \phi_R +
  Y_I^\top \phi_I, \\
  Z = (Y'\phi)_I &= ((Y_R + i Y_I)'(\phi_R + i\phi_I))_I = Y_R^\top \phi_I +
  Y_I^\top \phi_R,
  \end{align*}
so both are $\mathbb{R}$-linear transformations of real Gaussian random
vector $(Y_R, Y_I)$. Therefore, the expectation of $(X,Z)$ is
  $\mathbb{E}[(X,Z)]=0$ and
the variance of $X$ is
\[
  \mathrm{var}(X) = B \mathrm{Cov}((Y_R, Y_I)) B^\top = \frac12
  BB^\top,
\]
where the matrix $B$ is
\[
  B =
  \begin{pmatrix}
    \phi_R^\top & \phi_I^\top \\
  \end{pmatrix},
\]
therefore we have $\mathrm{var}(X) = \frac12 \phi'\phi$. We can similarly verify,
that $\mathrm{var}(Z) = \frac12 \phi'\phi$. Since $\mathbb{E}
(Y'\phi)'(Y'\phi) = \mathrm{var}(X) + \mathrm{var}(Z) = \phi'\phi$, we
see that the covariance of $Y$ is complex identity. We used the real
version to make the calculation easier.

\end{remark}

\subsection{Characteristic functions}
\label{sec:char}

\newcommand{\wY}{\widetilde{Y}}
\newcommand{\wphi}{\widetilde{\phi}}
\newcommand{\EW}{\mathbb{E}}
\newcommand{\CC}{\mathbb{C}}
\newcommand{\Rea}{\mathrm{Re}}
\newcommand{\Cov}{\mathrm{Cov}}
\newcommand{\KK}{\mathbb{K}}
\newcommand{\abs}[1]{\lvert #1 \rvert}
\newcommand{\dual}[2]{\langle\,#1\,,\,#2\,\rangle}

Since the complex Gaussian random vectors is defined as an affine
transformations of complex Gaussian white noise and the $k$-dimensional complex Gaussian
white noise is isomorphic with scaled $2k$-dimensional real white noise,
we can define the characteristic function via the following idea.

If $Y$ is a discrete $k$-dimensional complex white noise, then $\wY =
(Y_R, Y_I) / \sqrt{2}$ is
discrete $2k$-dimensional real white noise and its characteristic
function is
\[
  J_{\wY} (\wphi) = \EW \exp(i(\wphi^\top \wY))
  = \EW \exp(i(\phi_R^\top Y_R + \phi_I^\top Y_I)/\sqrt{2})
  = \EW \exp(i \Rea(Y' \phi)),
\]
where again $\Rea(\cdot)$, denotes the real component.
\begin{defn} (Characteristic function of complex Gaussian random vector)
  Assume $X := (X_1,\ldots,X_n)$ is a
  complex finite dimensional random vector. The function
\begin{equation*}
  J_X(\phi) = \EW \exp(i \Rea(X' \phi)),
\end{equation*}
where $\phi \in \mathbb{C}^n$ is the characteristic function of complex
Gaussian random vector.
\end{defn}

Note that via isomorphicity, the characteristic function fully
determines the distribution \cite{VIB98}.

\begin{prop}
  The characteristic function of discrete $k$-dimensional complex
  white noise $Y$ is
  \[
    J_{Y}(\phi) = \exp(-\frac{1}{4} \phi' \phi),
		= \exp(-\frac{1}{4} \abs{\phi}^2),
  \]
  where $\abs{\phi}^2 = \phi' \phi = \abs{\phi_1}^2 + \dots + \abs{\phi_k}^2$.
\end{prop}
\begin{proof}
   This follows with a straightforward computation. The $\CC$-covariance
   $\Cov(Y)$ of $Y = Y_R + i Y_I$ is by definition $\frac12 I_{\CC}$, so $Y_R$
   and $Y_I$ are independent and $\Cov(Y_R) = \Cov(Y_I) = \frac12
   I_{\bbR}$. Therefore
  \[
    \begin{split}
    J_{Y}(\phi) &= \EW \exp(iY_R^\top \phi_R) \EW \exp(iY_I^\top \phi_I)
		= \exp(-\frac14 \phi_R^\top \phi_R)\exp(-\frac14 \phi_I^\top \phi_I) \\
		&= \exp(-\frac14 \abs{\phi}^2).
    \end{split}
  \]

\end{proof}
\begin{prop}
  The characteristic function of $X = AY + \mu$, where $Y$ is
  $k$-dimensional complex white noise, $A \in \CC^{n \times k}$ and $\mu
  \in \CC^n$ is
  \[
    J_{X}(\phi) = \exp(i\Rea(\mu'\phi) -\frac{1}{4} \phi' \Sigma \phi),
  \]
  where $\Sigma = A A'$ is an self-adjoint matrix in $\CC^{n \times n}$.
\end{prop}

\begin{proof}
  Since $i\Rea(X' \phi) = i \Rea(\mu' \phi) + i \Rea((AY)'\phi)$, we may
  assume that $\mu = 0$ without a restriction.  Since $(AY)'\phi =
  Y'A'\phi = Y' \psi$, where $\phi = A'\phi$, the
  previous proposition gives that
  \[
    J_{X}(\phi) = J_{Y}(\psi) = \exp(-\frac{1}{4} \psi' \psi)
		= \exp(-\frac{1}{4} (A'\phi)' A \phi)
		= \exp(-\frac{1}{4} \phi' \Sigma \phi),
  \]
  which proves the claim.
\end{proof}
\begin{cor}
  The characteristic function of a complex Gaussian vector $X$ is
  \[
    J_{X}(\phi) = \exp(i\Rea(\EW(X)'\phi) -\frac{1}{2} \phi'
    \mathrm{Cov}(X) \phi),
  \]
  and the expectation and the complex covariance fully determine the
  distribution.
\end{cor}
\begin{proof}
This follows from previous results and the fact that $\Cov(Y) = \frac12
I$ for the complex white noise.
\end{proof}

\subsection{Densities for complex Gaussian vectors}
\label{sec:den}

By stating the density of the complex Gaussian vector $X$ we mean the
non-negative function $f \ge 0$ such that
\[
  \PP( X \in A ) = \int_{\CC^n} [\, x \in A \,] f(x) \dd x,
\]
where the integral is understood as a Lebesgue (volume) integral on
$\bbR^{2n}$. Note that not every complex Gaussian vector has a density
in this sense. However, every non-zero complex Gaussian vector has a
$\mathbb C$-affine subspace (potentially of lower dimension) of $\C^n$
such that the distribution is supported on this subspace and relative to
that the subspace it has a density.
The complex white noise itself has a density in this sense.

In order to extend this to other complex Gaussian vectors, we first
consider the orthogonal and unitary transformations. These are given
through the following propositions.

\begin{prop}
  The density function of discrete $k$-dimensional complex
  white noise $Y$ is
  \[
    f_{Y}(z) = \pi^{-n} \exp(-z' z)
	     = \pi^{-n} \exp(-\abs{z}^2),
  \]
  for every $z \in \CC^k$.
\end{prop}
\begin{proof}
  Since $Y$ is isomorphic to $\bbR^{2k}$-dimensional scaled white noise
  $(Y_R, Y_I)$ and the latter has a density on $\bbR^{2k}$ since it is a
  vector of $2k$ independent Gaussian random variables with zero mean
  and $\frac12$ variance. Therefore
  \[
    \begin{split}
    f_{(Y_R, Y_I)}(z_R, z_I)
      & = \prod_{j = 1}^n
	    (2\pi (1/2))^{-\frac12}
	    (2\pi (1/2))^{-\frac12}
	    \exp\Bigg(-\frac{(z_R)_j^2 + (z_I)_j^2}{2 \cdot (\frac12)}\Bigg) \\
      &	= \pi^{-n} \exp(-(z_R^\top z_R + z_I^\top z_I)) \\
	& = \pi^{-n} \exp(-z'z).
    \end{split}
  \]
\end{proof}

\begin{prop}
  Suppose $U \in \CC^{k \times k}$ is a unitary and $Y$ is a
  $k$-dimensional Gaussian random vector with density. Then $X = UY$
  also has density and  its density is given by
  \[
    f_{X}(z) = f_{Y}(U'z),
  \]
  for every $z \in \CC^k$.
\end{prop}
\begin{proof}
  This follows from the isomorphicity and the general transformation
  rule, since $U'$ is the inverse matrix of $U$ and the Jacobian
  determinant of the isomorphich copy of $U'$ is identically one, since
   \[
    \mathcal J_{\CC}(U')
       = \mathrm{det}
       \begin{pmatrix}
         U_R & -U_I \\
         U_I &  U_R \\
       \end{pmatrix}
       = \mathrm{det}(U_R^\top U_R + U_I^\top U_I)
       = \mathrm{det}((U'U)_R)
       = 1.
   \]
\end{proof}

\begin{prop}
  Suppose $U \in \CC^{k \times k}$ is a diagonal matrix $U =
  \mathrm{diag}(\lambda_1, \dots, \lambda_k)$ and $Y$ is discrete
  $k$-dimensional complex white noise $Y$. Then $X = UY$ has a density
  if and only if the determinant $D = \lambda_1 \dots \lambda_k \ne 0$.
  In this case it is given by
  \[
    f_{X}(z) = \abs{D}^{-1} f_Y(U^{-1}z),
  \]
  for every $z \in \CC^k$.
\end{prop}
\begin{proof}
  Let us first assume $D \ne 0$. In this case $X_j = \lambda_j Y_j$ for
  each $j = 1, \dots, k$. Moreover, the random variables $X_1, \dots,
  X_k$ are independent. This implies that each $X_j$ has a density
  function and the joint density is the product of the densities.

  Each $Y_j = \lambda_j^{-1} X_j$ which is isomorphic to $2$-dimensional
  real linear transformation: therefore,
  \[
    f_{X_j}(z_j)
		= \sqrt{\abs{\lambda_j^{-1}}^2} f_{Y_j}(z_j / \lambda_j)
		= \abs{\lambda_j}^{-1} f_{Y_j}(z_j / \lambda_j).
   \]
   The isomorphicity is inside the first identity, since the Jacobian
   determinant is
   \[
   \mathrm{det}
     \begin{pmatrix}
       (\lambda_j^{-1})_R & -(\lambda_j^{-1})_I \\
       (\lambda_j^{-1})_I &  (\lambda_j^{-1})_R \\
     \end{pmatrix}^{1/2}
    = \abs{\lambda_j^{-1}}^2,
   \]
  The claim follows by taking the products.

  If $D = 0$, then at least one the $\lambda_j$'s is zero. Without a
  loss of generality, we can for simplicity assume that $\lambda_1 = 0$.
  Then $Y = (0, Y_2, \dots, Y_k)$ and hence $Y$ is supported on a
  hypersurface of at most $k-1$ complex dimensions. This already implies
  that the density cannot exist.
\end{proof}

\begin{prop}
  Suppose $A \in \CC^{n \times n}$ is a matrix, $Y$ is discrete
  $n$-dimensional complex white noise $Y$ and $\mu \in \C^n$.
  The complex Gaussian vector $X = AY + \mu$ has a density if and only
  if $A$ is invertible. When $A$ is invertible, it is given by
  \[
    f_{X}(z) = \pi^{-n} \abs{\mathrm{det}(B)}^{-1/2} \exp(- (z-
    \mu)'B^{-1} (z-\mu)),
  \]
  for every $z \in \CC^n$, where $B = AA'$.
\end{prop}
\begin{proof}
  Without a restriction, we can assume $\mu = 0$. The
  matrix $B$ is self-adjoint, since $B' = (AA')' = AA' = B$, so it has
  a spectral decomposition $B = U\Lambda U'$ and a self-adjoint square
  root $\sqrt B := U \sqrt{\Lambda} U'$, i.e. $(\sqrt{B})' = \sqrt{B}$ and
  $(\sqrt B)^2 = B$. Note that $\mathrm{det}(\Lambda) = \mathrm{det}A$
  so the invertibility encoded into the diagonal matrix.

  Let $Z = \sqrt{B}Y$. The characteristic function of $Z$ is
  \[
    \begin{split}
    J_Z(\phi) &= \exp(-\frac14 \phi' \sqrt B (\sqrt B)' \phi) \\
              &= \exp(-\frac14 \phi' B \phi)
               = \exp(-\frac14 \phi' AA' \phi) \\
              &= J_X(\phi),
    \end{split}
  \]
  so $Z$ and $X$ are identically distributed. Therefore, $X$ has a density exactly when
  $Z$ has a density and in that case $f_X = f_Z$. Moreover, since $Z =
  U \sqrt{\Lambda} U' Y$, we moreover see that $Z$ and $U \sqrt{\Lambda}
  Y$ are identically distributed. This shows that
  \[
    f_{\sqrt{\Lambda}Y}(z)
    = \pi^{-n} \abs{D}^{-1} \exp(- (\sqrt{\Lambda}^{-1}z)'
					   (\sqrt{\Lambda}^{-1}z))
    = \pi^{-n} \abs{D}^{-1} \exp(- (z'{\Lambda}^{-1}z)),
  \]
  where $D = \mathrm{det}(B)$
  and thus
  \[
    f_Z(z)
    = \pi^{-n} \abs{D}^{-1} \exp(- ((U'z)'{\Lambda}^{-1}U'z))
    = \pi^{-n} \abs{D}^{-1} \exp(- (z' B^{-1}z)),
  \]
  which proves the claim.
\end{proof}
Now one can write the previous result directly with the general transformation
rule, but then the calculation of the determinant is more involved since
we cannot use the independence.
\begin{cor}
  If the covariance of a complex Gaussian $n$-dimensional vector $X$ is invertible, then
  $X$ has a density which is given by
  \[
    f_{X}(z) = (2\pi)^{-n} (\mathrm{det}(\Cov(X)))^{-1/2}
		 \exp(-\frac12 (z - \EW(X))'\Cov(X)^{-1}(z - \EW(X)),
  \]
  for every $z \in \C^n$.
\end{cor}
\begin{proof}
  When $X$ is discrete $n$-dimensional complex white noise, the $\Cov(X)
  = I_{\CC}/2$, so $(\mathrm{det}(\Cov(X)))^{-1/2} = 2^n$ and
  therefore
  \[
    \pi^{-n} = (2\pi)^{-n}(\mathrm{det}(\Cov(X)))^{-1/2},
  \]
  and
  \[
    \exp(-z'z) = \exp(-\frac12 z '\Cov(X)^{-1}z),
  \]
  so the claim holds for the discrete complex white noise. The remaining
  case follows from the previous proposition.
\end{proof}
\section{Infinite-dimensional analysis}
\label{sec:inf_analysis}
In this Appendix we extend the results of the previous section towards
the infinite dimensional case, where the underlying spaces are taken to be the rapidly decreasing functions
$\mathscr S(\C^n)$ (or the compactly supported test functions
$\mathscr D(\Omega)$) and their dual spaces $\mathscr S'(\C^n)$ of
tempered distributions (or the distributions $\mathscr D'(\Omega)$).
In particular  these can be done on the spaces of linear
operators $L(\mathscr S(\C^n), \mathscr S'(\C^n))$  between the dual spaces. For the time
being we will denote these as $\mathcal X_{\CC}$ and $\mathcal X'_{\CC}$
only to indicate that these are $\C$-linear vector spaces with
regularity in the topology, such that we can rigorously define the
concepts. In particular this appendix concludes the result of Theorem \ref{thm:main}.

By defining a Gaussian random object on $\mathcal X_{\CC}$ as
generalized Gaussian random variable $X \colon (\Omega, \mathscr F, \PP) \to (\mathcal
X'_{\CC}, \mathscr B(\mathcal X'_{\CC})$ via
\[
  \omega \mapsto (\phi \mapsto \dual{\phi}{X(\omega)}_{\mathcal X_{\CC} \times
  \mathcal X'_{\CC}}).
\]
We will drop the spaces from the dual action for simplicity.
We define the complex Gaussian noise as $Y$ on the underlying structure
as such that for every finite collection of ``test functions'' $\phi_1,
\dots, \phi_n$ the random object
\[
  Z := (\dual{\phi_1}{\overline Y}, \dots, \dual{\phi_n}{\overline Y}),
\]
is a complex Gaussian vector $n$ $\CC$-dimensions. Moreover, the
$\CC$-expectation $\EW(Z)$ of $Z$ is (isomorphic) to zero vector and
$\Cov(Z)$ is isomorphic to a $\CC^{n \times n}$-matrix
\[
  \begin{pmatrix}
    \mbox{$\frac12$} \dual{\phi_j}{\overline{\iota \phi_i}}
  \end{pmatrix}_{i,j},
\]
where $\iota \colon \mathcal X \to \mathcal X'$ is the natural embedding
of the ``test function'' space into its dual space. In order to proceed we first
need to ``mimic'' the definitions, but in infinite dimensions.
\begin{defn}
  Suppose $X$ is a $\mathcal X'$-valued random object. It has an
  \emph{expectation} $\EW X \in \mathcal X'$ if the following system of equations makes sense
  and has a unique solution
  \[
    \dual{\phi}{\overline{\EW(X)}} = \EW{\dual{\phi}{\overline X}},
  \]
  for every $\phi \in \mathcal X$.
\end{defn}
\begin{defn}
  Suppose $X$ is a $\mathcal X'$-valued random object. It has a
  \emph{covariance} $\Cov(X) \in L(\mathcal X, \mathcal X')$, if it has an expectation, the following
  system of equations makes sense and has a unique solution
  \[
    \dual{\phi}{\overline{\Cov(X) \phi}} =
    \EW \abs{\dual{\phi}{\overline W}}^2,
  \]
  for every $\phi \in \mathcal X$ and where $W = X - \EW X$.
\end{defn}
\begin{defn}
Suppose $X$ is a $\mathcal X'$-valued random object. The characteristic
function of $X$ is a mapping $J_X \colon \mathcal X \to \CC$ given by
\[
  J_{X} (\phi) = \EW \exp(i \Rea(\dual{\phi}{\overline X})).
\]
\end{defn}

We can verify that complex white noise $Y$ has the expectation $0 \in
\mathcal X'$ and its covariance $Y$ is $\Cov(Y) = \frac12 \iota$ which we will
later (incorrectly) call $\frac 12 I$ even though it is not the identity in that sense, it
would preserve the space. We can define the general complex Gaussian
object on $\mathcal X'$ exactly as before.

\begin{defn} (Complex Gaussian object)  Assume
  $X$ is a $\mathcal X'$-valued random object. We
  say $X$ is a complex Gaussian random object if it can be expressed in
  the form
\begin{equation}
X = \mu + A Y,
\end{equation}
  where $\mu \in \mathcal X'$, $A \in L(\mathcal Y', \mathcal X')$ and
  $Y$ is a $\mathcal Y'$-valued complex white noise.
\end{defn}

The main results generalize nearly verbatim, which are provided through the following
propositions,

\begin{prop}
  Assume that $X$ is a $\mathcal X'$-valued complex Gaussian object.
  Suppose $A \in L(\mathcal X', \mathcal Z')$ and $\mu \in \mathcal{Z}'$. Then
  $Z = \mu + AX$ is a $\mathcal Z'$-Gaussian random object with
  expectation
  \[
    \mathbb{E}(Z) = \mu + \mathbb{E}(X).
  \]
  It has covariance
  \[
    \mathrm{Cov}(Z) = A \mathrm{Cov}(X) A',
  \]
  where $A' = \overline{A}^* \in L(\mathcal Z, \mathcal X)$.
\end{prop}

\begin{prop}
  \label{appendix:propB6}
  The characteristic function of a complex Gaussian $\mathcal X'$-valued
  random object is
  \[
    J_{X}(\phi) = \exp(i\Rea(\dual{\phi}{\overline {\EW(X)}}) -
    \mbox{$\frac12$}
  \dual{\phi}{\overline{\mathrm{Cov}(X) \phi}}),
  \]
  and the expectation and the complex covariance fully determine the
  distribution.
\end{prop}

\subsection{Connection to radar equation}
\label{ssec:proof}

In this section we prove the Theorem \ref{thm:main}
Let's recall the radar equation~\eqref{eq:og} that was written as
\begin{equation*}
  z^q(t) = \int_0^1 \epsilon^q(t-r) \,\mu^q(dr) + \sqrt{T}\xi^q(t).
\end{equation*}
In order to be precise, this should be understood as a cyclic convolution
\begin{equation*}
  z^q = \epsilon_q * \mu^q + \sqrt T \xi^q,
\end{equation*}
\newcommand{\given}{\,|\,}
where \emph{given the covariance stucture} of the $\mu^q$, then $z^q,
\mu^q, \xi^q \in \mathcal X'$ are complex Gaussian $\mathcal X'$-valued
random objects and $\mathcal X' = \mathscr D'( \mathbb T; \C )$, the
$\mathbb T$ standing for the torus formed out of the interval $[0,1]$.

More precisely, we assume that the conditional distribution of $\mu^q$
given its covariance is known to be $X$, then $\mu^q \given X$ is a
complex Gaussian $\mathcal X'$-valued random object with zero mean and
\emph{random but given} covariance $X$. Writing $A_q \eta = \epsilon_q *
\eta$ we see that \emph{provided} the convolution makes sense $A_q$ is a
linear mapping form $\mathcal X'$ to $\mathcal X'$. Therefore, the
conditional characteristic function of $z^q$ is
\[
J_{z^q \given X}(\phi)
= \exp(-\mbox{$\frac12$} \dual{\phi}{\overline{A_q X A_q'\phi}}
       - \frac T2\abs{\phi}^2).
\]
Note that this is an extension of the simplified model. In order to
proceed, we assume that the covariance operators is parametrized. More
specifically,
\begin{equation}
\label{eq:lemmaB1}
  X = X(\sigma^2) = \phi \mapsto \sum_{j = 1}^N \sigma_j^2
\dual{\phi}{\iota \chi_j} \iota \chi_j,
\end{equation}
where $\{\chi_j\}_{j=1}^N$ form a periodic, smooth partition of unity normalized
in the $L^2$-sense.
This turns the bilinear form in the characteristic function into a
bilinear matrix form. This corresponds to the idea that the
autocovariance function is ``piecewise constant'', with $\chi_j$ acting
like a smooth indicator function. We will assume that the set
$\{\chi_j\}_{j=1}^N$ is known and the parameter vector $\sigma^2 = (\sigma_1^2,
\dots, \sigma_N^2)$ is the unknown replacing the full covariance
operator $X$.
For this special case, the conditional characteristic (given $\sigma^2$)
is
\[
J_{z^q \given \sigma^2}(\phi)
= \exp(-\mbox{$\frac12$} \dual{\phi}{\overline{A_q X(\sigma^2) A_q'\phi}}
       - \frac T2\abs{\phi}^2).
\]
With a straight forward calculation (recalling $A_q \eta = \epsilon_q *
\eta$ is understood as a mapping $\mathcal X' \to \mathcal X'$ and its
dual as a mapping $\mathcal X \to \mathcal X$), we see that
\[
  \dual {\phi}{\overline{A_q X(\sigma^2) A_q'\phi}}
  = \sum_{j = 1}^N \sigma_j^2 \abs{\dual{\phi}{A_q \iota \chi_j}}^2.
\]
Using $\phi = \phi_1 \pm \phi_2$ and summing up the previous identity implies
\begin{align*}
  \dual {\phi_1}{\overline{A_q X(\sigma^2) A_q'\phi_2}}
  &=
  \sum_{j = 1}^N \sigma_j^2 \dual{\phi_1}{A_q \iota \chi_j}
  \overline{\dual{\phi_2}{A_q \iota \chi_j}}, \\
  &=
  \sum_{j = 1}^N \sigma_j^2 \dual{\phi_1}{A_q \iota \chi_j}
  \overline{\dual{\phi_2}{A_q \iota \chi_j}}.
\end{align*}
Therefore, if we use a discrete dimensional complex Gaussian
\begin{equation}
\label{eq:lemmaB1_2}
Y_q = (z^q(\phi_1), \dots, z^q(\phi_M)),
\end{equation}
as a discrete observation from the measurement device, then
\[
  J_{Y_q \given \sigma^2}(\phi) =
    \exp(i\Rea(\EW(Y_q \given \sigma^2)'\phi) -\frac{1}{2} \phi'
    \mathrm{Cov}(Y_q \given \sigma^2) \phi).
\]
Linearity implies that
\[
\EW(Y_q \given \sigma^2) = 0,
\]
therefore,
\[
\phi'\mathrm{Cov}(Y_q \given \sigma^2) \phi
= \sum_{i,j = 1}^M \EW\dual{\phi_i}{\overline{z^q}}
\dual{\overline{\phi_j}}{z^q}).
\]
Using complex polarization, namely by calculating
\[
  \EW\abs{\dual{(\phi_i + \rho \phi_j)}{\overline{z^q}}}^2
  = \dual{(\phi_i + \rho \phi_j)}
	 {\overline{(A_q X(\sigma^2) A_q' + T)(\phi_i + \rho \phi_j)}},
\]
for $\rho \in \{\,1, -1, i, -i\,\}$ we find that
\begin{equation}
\label{eq:lemmaB1_3}
  \begin{split}
   \EW(\dual{\phi_i}{\overline{z^q}} \dual{\overline{\phi_j}}{z^q})
    &= \dual{\phi_i} {\overline{(A_q X(\sigma^2) A_q' + T)\phi_j}} \\
    &= \sum_{k = 1}^N \sigma_k^2 \dual{\phi_i}{A_q \iota \chi_k}
    \overline{\dual{\phi_j}{A_q \iota \chi_k}}\\
    &= \sum_{k = 1}^N \sigma_k^2
    {\phi'_j}{\overline{A_q \iota \chi_k}}
    (\overline{A_q \iota \chi_k})'{\phi_i} + T\phi'_j\phi_i \\
    &= \phi'_j \big(\sum_{k = 1}^N \sigma_k^2
    \overline{A_q \iota \chi_k \chi'_k \iota' A_q'} + T\big)
    \phi_i.
  \end{split}
\end{equation}
\textcolor{black}{Interpreting this generalized covariance operator as an complex
covariance operator of complex Gaussian vector, the density of $Y_q
\given \sigma^2$ is as a function of $\sigma^2$
seen to be proportional to an affine transform of the inverse Wishart
distribution.
\begin{lem}
\label{lemmaB1}
Suppose we have a known smooth, periodic partition of unity normalized in the
$L^2$-sense: $\{\chi_j\}_{j=1}^N$. Suppose $\sigma^2$ is the parameter
vector $\sigma^2 = (\sigma_1^2, \dots, \sigma_N^2)$ and the structure
function $X(\sigma^2)$ is defined with the
equation~\eqref{eq:lemmaB1}. Let $Y_q$ be the finite dimensional
marginal of the signal defined by~\eqref{eq:lemmaB1_2} with covariance
matrix $\Sigma$ defined by~\eqref{eq:lemmaB1_3}.
If $\Sigma_M$ follows the inverse Wishart distribution $\Sigma_M \sim
\mathcal{W}^{-1}(\Psi_M, \nu_M)$, then the
posterior distribution $\Sigma_M$ given the finite dimensional marginal
$Y_q$ of the signal follows the inverse Wishart distribution $\Sigma_M |
  Y_q \sim \mathcal{W}^{-1}(Y_q Y_q' + \Psi_M, \nu_M + 1)$
\end{lem}
\begin{proof}
  This follows by the above combining those with the results
  of~\cite{PS80}.
\end{proof}
\begin{remark}
If we increase the dimension $M$ of the finite marginal of the signal
and the number of parameters $N$ at the same time, we can invert the
affine transform between $\Sigma_M$ and and $\sigma^2$, and so posterior
distribution of $\sigma^2$ is seen to be an affine transform of inverse
Wishart distribution. Since this increases both the dimension of the
matrix $\Psi_M$ and the degrees of freedom, the interpretation of prior
could be done in terms of consistent families of inverse Wishart
distributions for the marginals. The previous lemma implies that the
posterior would still belong to the same consistent family.
\end{remark}
}
In the special case of Theorem~\ref{thm_yada41}, the $\abs{\sigma}^2$ is
constant and we can use a special smooth partition of unity that is
obtained with a single $\chi_1$ so that the all the others are periodic
translates of this $\chi_j = \tau^j (\chi_1)$ with
$\tau^j$ representing the $j^{\text{th}}$ iterate of the single
translate operation and which are rescaled to correspond to the
discretization of the measured signal. Moreover, since translation commute with
convolutions, we see that covariance operator for the discretization of
the following quadratic form 
\begin{equation}
\label{eq:quad}
  \phi
  \mapsto \int_{\mathbb{T}} |\,\sigma\,|^2 |\, \widehat
  \epsilon_{j} \,|^2(t) |\,\widehat{\phi}(t)\,|^2 \mathrm{d}t.
\end{equation}
where $\mathbb{T}$ denotes the one-dimensional torus that is isomorphic
with the half-open interval $[0, 2\pi)$.

\end{document}